\documentclass[letterpaper, 10 pt, conference]{ieeeconf}  

\IEEEoverridecommandlockouts                              

\usepackage{hyperref}
\usepackage{graphicx,color}
\graphicspath{{./IMG/}{images/}}
\usepackage{amsmath}
\usepackage{amssymb}
\usepackage{mathrsfs}
\usepackage{subfigure}
\usepackage{url}
\usepackage{booktabs}
\usepackage{array}
\usepackage[table]{xcolor}
\usepackage{mathtools} 		
\usepackage{epstopdf}
\usepackage{cite}
\usepackage[vlined,ruled]{algorithm2e}
\usepackage{upgreek}
\usepackage{dsfont}
\usepackage{comment}
\usepackage{algpseudocode}
\let\labelindent\relax
\usepackage{enumitem}  

\newtheorem{theorem}{Theorem}[section]
\newtheorem{lemma}[theorem]{Lemma}

\newtheorem{problem}{Problem}
\newtheorem{remark}{Remark}
\newtheorem{assumption}{Assumption}

\newcommand{\mc}{\mathcal}

\newcommand{\real}{\mathbb{R}}

\newcommand*{\QEDB}{\hfill\ensuremath{\square}}
\newcommand*{\QEDBL}{\hfill\ensuremath{\blacksquare}}
\newcommand\oprocendsymbol{\hbox{$\square$}}
\newcommand\oprocend{\relax\ifmmode\else\unskip\hfill%
\fi\oprocendsymbol}

\newcommand{\obj}{\tilde{\Phi}}

\title{\LARGE \bf
Two-point Random Gradient-free Methods for \\
Model-free Feedback Optimization}

\author{Amir Mehrnoosh and Gianluca Bianchin\thanks{
The authors are with ICTEAM Institute and the Department of Mathematical Engineering at UCLouvain, Belgium. A. Mehrnoosh is supported by F.R.S.-FNRS.
\{\href{mailto:amir.mehrnoosh@uclouvain.be}{\texttt{amir.mehrnoosh}},
    \href{mailto:gianluca.bianchin@uclouvain.be}{\texttt{gianluca.bianchin\}@uclouvain.be}.}
}}

\begin{document}

\maketitle
\thispagestyle{empty}
\pagestyle{empty}

\begin{abstract}

Feedback optimization has emerged as a promising approach for 
optimizing the steady-state operation of dynamical systems 
while requiring minimal modeling efforts. Unfortunately, most 
existing feedback optimization methods rely on 
knowledge of the plant dynamics, which may be difficult to obtain or estimate in practice. In this paper, we introduce a novel randomized two-point gradient-free feedback optimization method, inspired by zeroth-order optimization techniques. Our method relies on function evaluations at two points to estimate the gradient and update the control input in real-time. We provide convergence guarantees and 
show that our method is capable of computing an $\epsilon$-stationary point for smooth, nonconvex functions at a rate $\mathcal{O}(\epsilon^{-1})$, in line with existing results for two-point gradient-free methods for static optimization. Simulation results validate the findings. 

\end{abstract}

\section{Introduction}

Feedback optimization (FO) is concerned with the problem of controlling dynamical systems to an optimal steady-state point, as characterized by a 
mathematical optimization problem~\cite{AJ-ML-PB:09}. Examples of the applicability of this framework include optimal scheduling in communication networks, resource scheduling in power grids, optimization of transportation systems, and operation of industrial control processes. Traditional numerical optimization methods~\cite{nesterov2018lectures} provide a systematic approach to make control decisions when an exact model of the plant to control is available, and led to a rich class of model-based FO methods~\cite{AJ-ML-PB:09,FB-HD-CE:12,MC-ED-AB:20,GB-JC-JP-ED:21-tcns,GB-JIP-ED:20-automatica,GB-DL-MD-SB-JL-FD:21,AH-SB-GH-FD:20}.
Yet, in real-world applications, an accurate model
of the plant to control is rarely available~\cite{bianchin2023online}, and thus implementing and ensuring the optimality of these methods remains challenging. 
To overcome these limitations, a model-free approach for FO has recently been proposed 
in~\cite{he2023model}, relying on a one-point 
residual-feedback gradient estimate~\cite{zhang2022new}. 
Unfortunately, as is well-known in the optimization literature (see~\cite{zhang2022new} for an insightful comparison), methods relying on one-point gradient estimates are unable to recover the rate of convergence of methods based on knowledge of the exact gradient~\cite{zhang2022new}.
Motivated by these limitations, in this paper, we propose a 
two-point random gradient-free method for feedback 
optimization. The proposed controller relies on two 
function evaluations of the plant performance to estimate a 
descent direction, combined with a random exploration step 
for the control input. To incorporate two function evaluations at each iteration, the controller is designed to operate at a slower timescale than that of the plant, in line with existing approaches on FO~\cite{GB-JC-JP-ED:21-tcns,AH-SB-GH-FD:20,GB-JIP-ED:20-automatica}.
We show that, for smooth nonconvex problems, the proposed 
two-point method computes an $\epsilon$-stationary point in
$\mathcal{O}(\epsilon^{-1})$ iterations, this outperforms existing single-point methods~\cite{he2023model}, which are characterized by a rate of $\mathcal{O}(\epsilon^{-3/2})$
for the same problem.

\textit{Related works.} Feedback optimization controllers have attracted significant interest due to their capacity to steer systems toward optimal steady states while effectively rejecting both constant and time-varying disturbances~\cite{AH-SB-GH-FD:20,GB-JC-JP-ED:21-tcns,GB-JIP-ED:20-automatica}. These approaches integrate numerical optimization into feedback control by leveraging real-time measurements to estimate gradients, thereby removing the reliance on accurate models of the plant and disturbances. In~\cite{chen2020model}, a fast-stable plant is considered as an algebraic steady-state map for the power flow application.
Particularly related to this work is the recent 
work~\cite{he2023model}, which is the first fully model-free 
method for feedback optimization, relying on a one-point 
residual-feedback gradient estimate~\cite{zhang2022new}.

Another key development in FO that emerged to address scalability and privacy issues in large-scale systems is distributed FO. When centralized approaches become impractical because of the system size or the need to keep internal states private, decentralized methods provide a viable alternative. Distributed gradient descent was first introduced in~\cite{nedic2009distributed} and further analyzed in~\cite{yuan2016convergence}, with other distributed optimization algorithms extensively explored in~\cite{yang2019survey}. Building on these foundations, the work~\cite{mehrnoosh2024distributed} and~\cite{carnevale2024nonconvex} propose distributed FO methods that integrate FO with distributed computation to make FO practical for distributed architectures.

The literature on (static) zeroth-order optimization \cite{liu2020primer} is also related to this work. 
Zeroth-order methods refer to a class of optimization techniques that estimate gradients using function evaluations instead of explicit sensitivity information. These methods, 
such as one-point feedback~\cite{flaxman2004online} and an improved version~\cite{chen2022improve}, one-point residual feedback~\cite{zhang2022new}, and two-point 
feedback~\cite{YN-VP:17, duchi2015optimal,AA-OD-LX:10,grapiglia2024worst} include randomized gradient-free techniques. They have demonstrated convergence properties comparable to first-order approaches that make them suitable for model-free optimization. Specifically, two-point gradient-free methods that estimate gradients using finite differences between function evaluations at two distinct points are of particular interest because of their better rate of convergence relative to other gradient-free methods. However, this comes at the cost of increased computational complexity, as they require two function evaluations per estimation.

\textit{Contributions.}
This work features two main contributions. First, 
we propose a two-point random gradient-free method for 
feedback optimization. In a net departure 
from~\cite{he2023model}, we utilize a two-point gradient 
estimate rather than a single-point one.  
Second, we provide a rigorous convergence analysis of our 
approach. Although our method requires two function 
evaluations to estimate a descent direction, its rate of 
convergence outperforms existing approaches based on a
single evaluation. 
Importantly, our analysis shows that the 
rate of convergence of our method recovers that of established two-point gradient methods in static optimization~\cite{YN-VP:17}.

\textit{Organization.}
The paper is structured as follows. Section~\ref{sec:problem_formulation} defines the problem focus of this work. Section~\ref{sec:control_design} introduces the proposed two-point random gradient-free controller. Section~\ref{sec:analysis} presents the main theoretical results, followed by numerical validation in Section~\ref{sec:simulations}. Finally, Section~\ref{sec:conclusions} concludes the paper.

\section{Problem formulation}

\label{sec:problem_formulation}

We consider plants that can be modeled by a discrete-time dynamical models of the form:
\begin{align} \label{eq:system}
    x_{t+1} &= f(x_t,u_t,d), \nonumber \\
    y_t &= z(x_t,d),
\end{align}
where $x_t \in \real^n$ is the system state at time $t$, $u_t \in \real^p$ the 
control input, $y_t \in \real^q$ the measured output, and $d \in \real^r$ models 
a deterministic but unknown constant disturbance. In this work, we are 
interested in an output regulation control problem; for this to be well-posed, 
we make the following assumption.

\begin{assumption}[\textbf{\textit{Properties of the plant}}]
\label{assum:system}
There exists a unique map $x_{\text{ss}}: \mathbb{R}^{p}\times \mathbb{R}^{r} \to \mathbb{R}^{n}$ such that $\forall u,d, f(x_{\textrm{ss}}(u,d),u,d)=x_{\textrm{ss}}(u,d)$; the mapping $u \mapsto x_{\textrm{ss}}(u,d)$ is globally $M_x$-Lipschitz in $u$, and the function $z(x,d)$ is globally $M_z$-Lipschitz in $x$. 
Moreover, for each $(u,d)$, the equilibrium point $x_{\text{ss}}(u,d)$ 
of~\eqref{eq:system} is globally exponential stable; that is, there exists 
$\beta,\tau > 0$ such that for any $x_0 \in \mathbb{R}^n$, the solutions 
of~\eqref{eq:system} with $u_t=u \ \forall t$ satisfy 
$\|x_t - x_{\text{ss}}(u,d)\| \leq \beta\|x_0 - x_{\textrm{ss}}(u,d)\|e^{-\tau t}$.
\QEDB\end{assumption}

In other words, Assumption~\ref{assum:system} guarantees the existence of a function
$x_{\text{ss}}$ that maps each input pair $(u,d)$ into the corresponding 
equilibrium state and that the equilibrium of~\eqref{eq:system} is globally 
exponentially stable. We note that these assumptions are common in output 
regulation problems~\cite{ED:76} as well as in feedback 
optimization approaches~\cite{GB-JC-JP-ED:21-tcns,AH-SB-GH-FD:20,carnevale2024nonconvex}.
Notice also that if the plant to be regulated is not asymptotically stable, 
\eqref{eq:system} shall be viewed as a pre-stabilized version of such a plant 
(e.g., stabilized via classical approaches based on output 
feedback~\cite{HKK:96}). In what follows, we let
\begin{equation}\label{eq:ssMap}
    h(u,d) \triangleq  z(x_{\text{ss}}(u,d),d).
\end{equation}

It follows from Assumption~\ref{assum:system} that~\eqref{eq:system} is globally input-to-state stable with 
respect to the input $(u_{t+1}-u_t)$ (see~\cite{ZPJ-YW:01}); 
that is, there exists\footnote{See~\cite{ZPJ-YW:01} for notation.} a $\mathcal{K} \mathcal{L}$-function 
$\beta: \mathbb{R}_{\geq 0} \times \mathbb{R}_{\geq 0} \longrightarrow \mathbb{R}$ and a $\mathcal{K}$-function 
$\gamma$ such that, for each input signal $u_t$ and each 
$x_0 \in \mathbb{R}^n$, it holds that
\begin{align}\label{eq:iss}
\Vert x_{t+1} - x_{\text{ss}}(u_t,d)\Vert \leq 
\max _{1 \leq j \leq t} & \{\sigma_1(\Vert x_{1} - x_{\text{ss}}(u_0,d)\Vert), \\
&\quad\quad\quad\quad\quad \sigma_2(\Vert u_j - u_{j-1}\Vert)\},\nonumber
\end{align}
for all $x_1\in\real^n,$ where $\sigma_1, \sigma_2$ are  two 
$\mathcal{K}$-functions.
Motivated by this, we make the following assumption.

\begin{assumption}[\textbf{\textit{Properties of $h(u,d)$}}]
\label{assum:funUtoY}
    There exists a $\mu \geq 0$ such that for any input $u_t \in \real^p$,
    \begin{align*}
        \Vert y_{t+1} - h(u_t,d) \Vert^2 \leq \mu.
    \end{align*}
    \QEDB
\end{assumption}

The quantity $\mu$ can be interpreted as an estimate for the 
speed of the dynamics of the plant~\eqref{eq:system}, capturing the rate at 
which~\eqref{eq:system} converges to its steady-state output. 
By comparison with~\eqref{eq:iss}, $\mu$  can be viewed
as an estimate for the right-hand side of~\eqref{eq:iss} 
(combined with the Lipschitz constant $M_z$). Notice that a 
uniform bound for the right-hand side of~\eqref{eq:iss} is a 
reasonable approximation when $\Vert x_{1} - x_{\text{ss}}(u_0,d)\Vert$ is bounded (in other words, the initial 
condition of~\eqref{eq:system} is close to the plant's 
steady-state), and the controller is sufficiently-slow, so 
that $\Vert u_j - u_{j-1}\Vert$  is bounded.
We leave a relaxation of this assumption as the scope of 
future works.

In this work, we study the problem of designing a controller that 
regulates~\eqref{eq:system} to the solution of the following optimal output 
regulation problem:
\begin{align} \label{eq:trueProblem}
    \begin{split}
        \min_{u,y} ~ &\Phi(u,y) \\
        \textrm{s.t.} ~ &y = h(u,d),
    \end{split}
\end{align}
where $\Phi: \real^p \times \real^q \to \real$ is a (possibly nonconvex) loss 
function modeling performance requirements for system inputs and outputs at 
steady-state. By substituting the constraint into the cost, 
\eqref{eq:trueProblem} can be rewritten as an unconstrained optimization problem:
\begin{align} \label{eq:unconProblem}
    \min_u ~ \tilde{\Phi}(u) \triangleq \Phi(u,h(u,d)).
\end{align}
We make the following assumption on the loss functions.

\begin{assumption}[\textbf{\textit{Properties of the optimization}}]
\label{assum:costs}
The function $\obj(u)$ is globally $L$-smooth, $M$-Lipschitz, and is bounded below by $\obj_{\text{low}}$. Moreover, the function $\Phi(u,y)$ is globally $M_\Phi$-Lipschitz in $y$.
        \QEDB
\end{assumption}

The assumptions on Lipschitz continuity in Assumption~\ref{assum:costs} are 
standard and largely satisfied in applications.

A standard approach to regulate~\eqref{eq:system} to the solution 
of~\eqref{eq:unconProblem} is as follows:
\begin{enumerate}[label=(S\arabic*), leftmargin=*]
\item Apply an optimization algorithm to determine an optimizer $u^\star$ 
of~\eqref{eq:unconProblem}    
\item Apply $u_t \equiv u^\star$ to~\eqref{eq:system}
\end{enumerate}
Unfortunately, such an approach is impractical because of two main limitations: 
\begin{enumerate}[label=(L\arabic*), leftmargin=*]
\item Solving~\eqref{eq:unconProblem} requires knowledge of $d,$ which in many 
practical applications is unknown
\item Solving~\eqref{eq:unconProblem} requires knowledge of the mapping $h(u,d)$ 
and hence of the full model of the plant~\eqref{eq:system} (precisely, of the 
functions $f(x,u,d),$ $z(x,d)$, and $x_{\text{ss}}(u,d)$ implicitly
through~\eqref{eq:ssMap}), which is impractical in many cases, as models are 
often unknown or inexact
\end{enumerate}

Motivated by these observations, in this work we study the following problem.

\begin{problem}\label{prob:main_objective}
Design a control algorithm for~\eqref{eq:system}, having access only to 
performance evaluations of the plant at each time (i.e., oracle evaluations of 
$t \mapsto \Phi(u_t,y_{t+1})$) and without any knowledge of the disturbance $d$ 
nor of the model of~\eqref{eq:system}, such that the inputs and outputs 
of~\eqref{eq:system} converge asymptotically to an optimizer of~\eqref{eq:trueProblem}.
\QEDB\end{problem}
\smallskip

We stress that we seek an algorithm that relies on function evaluations of the 
map $\Phi(u,y)$ (cf.~\eqref{eq:trueProblem}) rather than $\tilde \Phi(u)$ 
(cf.~\eqref{eq:unconProblem}) because the mapping $h(u,d)$ is assumed to be 
unknown (see limitation(L2)), thus seeking a method that is entirely 
`model-free.'

\section{The two-point Random Gradient-free Method for Feedback Optimization}
\label{sec:control_design}

In this section, we propose a two-point random gradient-free method for feedback 
optimization that achieves the objectives set forward in 
Problem~\ref{prob:main_objective}. The method is summarized in 
Algorithm~\ref{alg:2PTRGF}. First, the input $u_t$ is applied to the plant and 
the corresponding control performance is evaluated through $\Phi(u_t, y_{t+1})$ 
(cf. line 2). Then, the control input is randomly perturbed around 
the current point $u_t$ (cf. line 3), applied to the plant, and the control 
performance is re-evaluated at such a perturbed point
$\Phi(u_{t+1}, y_{t+2})$ (cf. line 4). By using these two function 
evaluations, a descent direction for the cost $g_t^\delta$ is estimated (cf. 
line 5) and, finally, the control decision is updated along this direction.
In the algorithm, the parameter $\eta>0$ is interpreted as the stepsize of the 
method and $\delta>0$ as the smoothing parameter modeling the magnitude of the 
perturbation.

\begin{algorithm} \label{alg:2PTRGF}
   \caption{Two-point random gradient-free feedback controller} 
   \begin{algorithmic}[1]
       \State {\bf Data:} $u_0  \in \real^p$, $x_0  \in \real^n$, $t = 0$, $\eta,\delta>0$
       \State Apply $u_t$ to \eqref{eq:system} and evaluate $\Phi(u_t, y_{t+1})$
       \State Set $u_{t+1} = u_t+\delta v_t$, $v_t \sim \mathcal{N}(0,I_p)$
       \State Apply $u_{t+1}$ to \eqref{eq:system} and evaluate $\Phi(u_{t+1}, y_{t+2})$
       \State Set $g_t^{\delta} = \frac{v_t}{\delta} \big( \Phi(u_{t+1},y_{t+2}) - \Phi(u_{t},y_{t+1}) \big)$ \label{line:gradEst}
       \State Set $u_{t+2} = u_t - \eta g_t^\delta$ \label{line:update}
       \State $t\gets t+2$, go to step $2$
   \end{algorithmic}
\end{algorithm}

\begin{remark}
Algorithm~\ref{alg:2PTRGF} can be seen as a variant of the random gradient-free 
two-point optimization method~\cite{YN-VP:17,AA-OD-LX:10,OS-17}, specifically 
adapted for use in FO.
This algorithm is characterized by two main innovative features with respect to~\cite{YN-VP:17,AA-OD-LX:10,OS-17}. First, each function evaluation 
$t \mapsto \Phi(u_t,y_{t+1})$ implicitly requires one state update of the 
plant~\eqref{eq:system} (namely, from $x_t$ to $x_{t+1}$), and thus estimating a 
descent direction (as in line 5 of the algorithm) requires a combination of two 
plant updates. Moreover, because the algorithm relies on function evaluations 
of $\Phi(u,y)$ (in place of $\tilde \Phi(u)$) for the limitations outlined in 
(L1)-(L2), the descent direction estimated by $g_t^\delta$ is an 
\textit{approximation}  of a descent direction for $\tilde \Phi(u).$ These two 
properties pose additional challenges for the closed-loop performance analysis, 
which will be addressed in the subsequent sections.
\QEDB\end{remark}
\smallskip

The closed-loop dynamics, when Algorithm~\ref{alg:2PTRGF} is used to 
control the plant~\eqref{eq:system} are, for\footnote{See the detailed 
description in Algorithm~\ref{alg:2PTRGF} for initializations and the 
special case $t=0.$} $t=1,2,3, \dots$, given by:
\begin{align}\label{eq:closed-loop}
x_{t+1} &= f(x_t,u_t,d), \quad  y_t = z(x_t,d), \\
u_{t+1} &= \begin{cases}
u_t+\delta v_t, \quad v_t \sim \mathcal{N}(0,I_p), & \text{ if } t=2,4,6, \dots,\\
u_{t-1} - \eta \frac{v_{t-1}}{\delta} g_{t-1}^\delta & \text{ if } t=1,3,5, \dots,\\
\end{cases}\nonumber
\end{align}
where $g_{t-1}^\delta = \big( \Phi(u_{t},y_{t+1}) - \Phi(u_{t-1},y_{t}) \big).$
The following result characterizes the control performance of 
Algorithm~\ref{alg:2PTRGF}, when applied as a feedback controller as 
in~\eqref{eq:closed-loop}.

\begin{theorem} \label{thm:objConvergence}
Suppose that Assumptions~\ref{assum:system}-\ref{assum:costs} hold. 
Fix $\epsilon_\Phi>0$, and assume that $\eta<1/8L(p+4)$ and 
$ \delta \leq \sqrt{2\epsilon_\Phi /Lp}$.
Then, the closed-loop system~\eqref{eq:closed-loop}, after $T>0$, iterations 
satisfies\footnote{A function $\psi (n)$ is said to be $\mathcal{O}\big(f(n)\big)$ if there exists a constant $C > 0$ and $n_0$ such that $\forall n \geq n_0$, $\vert \psi(n) \vert \leq C\vert f(n) \vert$.}:
    \begin{align} \label{eq:convBound}
        & \frac{1}{T} \sum_{k=0}^{T-1} E_{v}[\|\nabla \tilde \Phi(u_{k})\|^{2}] 
        = \mathcal{O}\bigg(\frac{1}{T\eta(1-\eta p)}\bigg) + \mathcal{O}\bigg(\frac{\eta \delta^2 p^3}{(1-\eta p)}  \bigg)
        \notag \\
        &  \hspace{2cm} + \mathcal{O}\bigg(\delta^2p^3\bigg) 
        +\mathcal{O}\bigg(\frac{\mu p \eta}{\delta^2 (1-\eta p)}\bigg),
    \end{align}
where $\epsilon_\Phi$ is the precision satisfying $\vert \obj_\delta(u) - \obj(u) \vert \leq \epsilon_\Phi$ and the expectation is with respect to $\{v_0, \dots , v_{T-1}\}$.~
\QEDB\end{theorem}

The proof of this result is presented in Section~\ref{sec:analysis}.
In other words, Theorem~\ref{thm:objConvergence} guarantees that the 
control sequence $u_k$, produced by the closed-loop 
system~\eqref{eq:closed-loop}, yields a sequence of gradient errors 
$\|\nabla \tilde \Phi(u_k)\|^2$ that is summable in expectation.
The upper bound in~\eqref{eq:convBound} depends on the various 
parameters of the optimization problem~\eqref{eq:trueProblem}, 
of the plant~\eqref{eq:system}, and of the algorithm; in particular, the first term decreases to zero at a rate $1/T$, the second and third terms can be made arbitrarily small by carefully choosing a suitable small 
$\eta$ and $\delta$, while the fourth term depends on the rate of 
convergence of the plant and can be reduced when $\mu$  is a tunable parameter. In fact, for a rapidly decaying system ($\mu$ close to zero), this term can be neglected.

\begin{remark}
In Theorem~\ref{thm:objConvergence}, the average second moment of the gradient of the $\obj(u)$ is the convergence measure. This measure is extensively used in the field of zeroth-order optimization when the problem is nonconvex~\cite{YN-VP:17,zhang2022new}. We say a solution $u$ is $\epsilon$-accurate if $E[\Vert \nabla \obj(u) \Vert^2] \leq \epsilon$. Note that finding a globally optimal solution for nonconvex problems is NP-hard~\cite{danilova2021recenttheoreticaladvancesnonconvex}. Hence, this measure serves as a starting point for analysis in this work. In addition, we need the Gaussian smooth approximation $\obj_\delta$ to be $\epsilon_\Phi$-close to the original objective function $\obj$, which requires $\delta \leq \sqrt{2\epsilon_\Phi /Lp}$ according to Lemma~\ref{lem:gaussianSmooth} in the Appendix.

\QEDB
\end{remark}

The following result gives an explicit way to select all parameters of the 
optimization method to ensure that the iterates of~\eqref{eq:closed-loop} 
converge to an $\epsilon$-stationary point of~\eqref{eq:unconProblem}.

\begin{theorem}
\label{thm:optimal_parameters_selection}
Suppose that Assumptions~\ref{assum:system}-\ref{assum:costs} hold. 
Fix $\epsilon, \epsilon_\Phi>0$, let $\eta = 1/16L(p+4)$,
    \begin{align}\label{eq:deltasquared}
        \delta^2 = \sqrt{\frac{4M_\Phi^2 \mu p (8p+33)}{L^2 (p+4) \big((p+6)^3 + (p+4)^2 \big)}},
    \end{align}
    and suppose that $\mu \leq \min\{\mu_1, \mu_2\},$ where
    \begin{align}\label{eq_mu1_mu2}
        \mu_1 &= \frac{(p+4)\epsilon^2}{16L^2 M_\Phi^2 p (8p+33)\big((p+6)^3 + (p+4)^2 \big)}, \notag \\
        \mu_2 &= \frac{(p+4)\big((p+6)^3 + (p+4)^2 \big) \epsilon_\Phi^2}{M_\Phi^2 p^3 (8p+33)}.
    \end{align}
Then, after $T \geq  2c_1/\epsilon$ iterations, the closed-loop 
system~\eqref{eq:closed-loop}  satisfies:
    \begin{align}
        \frac{1}{T} \sum_{k=0}^{T-1} E_{v}[\|\nabla \tilde \Phi(u_{k})\|^{2}] \leq \epsilon,
    \end{align}
where $c_1 = 128L(p+4)(\tilde \Phi_\delta(u_0) - \tilde \Phi_{low})$
and the expectation is taken with respect to $\{v_0, \dots , v_{T-1}\}$.
\end{theorem}

The proof of this result is presented in Section~\ref{sec:analysis}.
In other  words, given a desired accuracy $\epsilon,$ 
Theorem~\ref{thm:optimal_parameters_selection} gives a method to select the 
parameters $\eta$ and $\delta$ so that~\eqref{eq:closed-loop} reaches an 
$\epsilon$-stationary point of~\eqref{eq:unconProblem}. Notice the choice of 
$\eta$ and $\delta$ made here are compatible with the range for these parameters 
given in Theorem~\ref{thm:objConvergence}. Particularly, by combining~\eqref{eq:deltasquared} with~\eqref{eq_mu1_mu2}, it 
follows that, to reach an $\epsilon$-stationary point, the smoothing parameter 
should be chosen $\delta^2 = \mc O (\frac{\max \{ \epsilon, \epsilon_\Phi\}}{L^2 p^3})$.

\begin{remark}
    Interestingly, the required iteration complexity for the proposed algorithm is of order $\mathcal{O}(p\epsilon^{-1})$, which is the same as the best complexity result for two-point zeroth-order feedback established in~\cite{YN-VP:17}, whereas the iteration complexity for the one-point residual feedback is of order $\mathcal{O}(p^3 \epsilon^{-3/2})$. Note that this comparison holds when the objective function is smooth and nonconvex~\cite{zhang2022new}.
\QEDB\end{remark}

\section{Convergence and Performance Analysis}
\label{sec:analysis}

In this section, we analyze the iterates of~\eqref{eq:closed-loop} and present 
the proofs of Theorems~\ref{thm:objConvergence}--\ref{thm:optimal_parameters_selection}.

First, we introduce the compact notation:
\begin{align}\label{eq:g_gtilde}
 \tilde g_\delta(u_t) &:= \frac{v_t}{\delta} \big(\tilde \Phi(u_t + \delta v_t) - \tilde \Phi(u_t)\big),\nonumber\\
 g_\delta(u_t) &:=
 \frac{v_t}{\delta} \big( \Phi(u_{t+1},y_{t+2}) - \Phi(u_{t},y_{t+1}) \big),
\end{align}
and observe that $\tilde g_\delta(u_t)$ models the two-point gradient 
estimator~\cite{YN-VP:17,AA-OD-LX:10,OS-17} based for the true function 
$\tilde \Phi(u_k)$ that we aim to minimize  (see~\eqref{eq:unconProblem}). 
Moreover, define the gradient estimator error:
\begin{align} \label{eq:errorNestOurs}
    e_t := g_\delta(u_t) - \tilde g_\delta(u_t).
\end{align}

\subsection{Instrumental results}
The following results are instrumental for the subsequent analysis.

\begin{lemma}[{\!\!{\cite[Lemma 1]{YN-VP:17}}}] \label{lem:nestRand}
    Let $v\in \mathbb{R}^{p}$ satisfy the standard multivariate normal distribution. Then,
    \begin{align*}
        E[\|v\|^t]\leq
        \begin{cases}
            p^{t/2}, & \text{if~} t\in [0,2], \\
            (p+t)^{t/2}, & \text{if~} t> 2.
        \end{cases}
    \end{align*}
\QEDB\end{lemma}

\begin{lemma} \label{lem:boundE}
    If Assumptions~\ref{assum:system}-\ref{assum:costs} hold, then
    \begin{align}
        E[\Vert e_t \Vert^2] \leq \frac{4M_\Phi^2\mu p}{\delta^2}.
    \end{align}
\QEDB\end{lemma}

\begin{proof}
By substituting~\eqref{eq:g_gtilde} into~\eqref{eq:errorNestOurs}:
\begin{align}
e_t &=  \frac{v_t}{\delta} \big[\Phi(u_{t+1},y_{t+2}) \nonumber \\
& \quad \quad\quad - \Phi(u_{t},y_{t+1}) - \big(\obj(u_{t+1}) - \obj(u_t)\big)\big].
    \end{align}
    Now, we take the $2$-norm from both sides of the inequality and use $(a+b)^2 \leq 2a^2 + 2b^2$. Noting that $\obj(u_t) = \Phi(u_t,h(u_t,d))$, we obtain
    \begin{align} \label{eq:aux4}
        \|e_t\|^2 &\leq \frac{2\|v_t\|^2}{\delta^2} \big\| \Phi(u_{t+1},y_{t+2}) -  \Phi(u_{t+1},h(u_{t+1},d))\big\|^2 \notag \\
        &+ \frac{2\|v_t\|^2}{\delta^2} \big\| \Phi(u_{t},y_{t+1}) -  \Phi(u_{t},h(u_t,d))\big\|^2, \notag \\
        &\overset{(a)}{\leq} \frac{2\|v_t\|^2}{\delta^2} M_\Phi^2 \big(\Vert y_{t+2} - h(u_{t+1},d)\Vert^2 \notag \\
        &\qquad\qquad\qquad + \Vert y_{t+1} - h(u_{t},d)\Vert^2 \big), \notag \\
        &\overset{(b)}{\leq} \frac{4\|v_t\|^2}{\delta^2} M_\Phi^2\mu,
    \end{align}
    where $(a)$ follows from the Lipschitz continuity of $\Phi(u,y)$ in $y$, and $(b)$ from Assumption~\ref{assum:funUtoY}. Taking the expectations of both sides of~\ref{eq:aux4}, and utilizing Lemma~\ref{lem:nestRand} completes the proof.
\end{proof}

Let $\{ a_k\}_{k=0}^\infty, a_k = 2k,$ denote the sequence of even, 
non-negative, integers. With a slight abuse of notation, in what follows, we 
will denote $a_k$ simply by $k$, such that $k+1$ is intended to denote 
$a_{k+1} = 2(k+1).$ Notice also that we will use the notation $k$ to explicitly 
distinguish it from the time index $t$ in \eqref{eq:closed-loop}.

\subsection{Proof of Theorems~\ref{thm:objConvergence}-\ref{thm:optimal_parameters_selection}}

We begin with the proof of Theorem~\ref{thm:objConvergence}. From 
Assumption~\ref{assum:costs}, we have
    \begin{align} \label{eq:aux}
        \tilde \Phi_\delta&(u_{k+1}) \notag \\
        &\leq \tilde \Phi_\delta(u_{k}) + \big<\nabla \tilde \Phi_\delta(u_{k}), u_{k+1}-u_k \big> + \frac{L}{2} \|u_{k+1}-u_k\|^2 \notag \\
        &\overset{(a)}{\leq} \tilde \Phi_\delta(u_{k}) - \eta \big<\nabla \tilde \Phi_\delta(u_{k}), \tilde g_\delta(u_k)+e_k \big> \notag \\
        &+ \frac{L\eta^2}{2} \|\tilde g_\delta(u_k)+e_k\|^2,
    \end{align}
    where $(a)$ follows from line 5 of Algorithm~\ref{alg:2PTRGF} and \eqref{eq:errorNestOurs}. By taking the expectations of both sides of \eqref{eq:aux} with respect to $v_{k}$ and using $(a+b)^2 \leq 2a^2 + 2b^2$ inequality on the third term, we have
    \begin{align}
        E_{v_k}&[\tilde \Phi_\delta(u_{k+1})] \notag \\
        &\leq \tilde \Phi_\delta(u_{k}) - \eta \big<\nabla \tilde \Phi_\delta(u_{k}), E_{v_k}[\tilde g_\delta(u_k)] \big> \notag \\
        &- \eta \big<\nabla \tilde \Phi_\delta(u_{k}), E_{v_k}[e_k] \big> + L\eta^2 E_{v_k}\big[\|\tilde g_\delta(u_k)\|^2\big] \notag \\
        &+ L\eta^2 E_{v_k}\big[\|e_k\|^2\big] \notag \\
        &\overset{(a)}{\leq} \tilde \Phi_\delta(u_{k}) - \eta \|\nabla \tilde \Phi_\delta(u_{k})\|^2 - \eta \big<\nabla \tilde \Phi_\delta(u_{k}), E_{v_k}[e_k] \big> \notag \\
        &+ L\eta^2 E_{v_k}\big[\|\tilde g_\delta(u_k)\|^2\big] + L\eta^2 E_{v_k}\big[\|e_k\|^2\big] \notag \\
        &\overset{(b)}{\leq} \tilde \Phi_\delta(u_{k}) - \eta \|\nabla \tilde \Phi_\delta(u_{k})\|^2 - \eta \big<\nabla \tilde \Phi_\delta(u_{k}), E_{v_k}[e_k] \big> \notag \\
        & + 4L\eta^2(p+4) \|\nabla \tilde \Phi_\delta(u_k)\|^2 + 3L^3\eta^2\delta^2(p+4)^3 \notag \\
        &+ L\eta^2 E_{v_k}\big[\|e_k\|^2\big]
    \end{align}
    where $(a)$ and $(b)$ follow from Lemma~\ref{lem:gradEstBound}($i$) and ($ii$) in Appendix, respectively. By rearranging terms, we obtain
    \begin{align}
        &\big( \eta - 4L\eta^2(p+4) \big) \|\nabla \tilde \Phi_\delta(u_k)\|^2 \notag \\
        &\leq \tilde \Phi_\delta(u_{k}) -  E_{v_k}[\tilde \Phi_\delta(u_{k+1})] - \eta \big<\nabla \tilde \Phi_\delta(u_{k}), E_{v_k}[e_k] \big> \notag \\
        &+ 3L^3\eta^2\delta^2(p+4)^3 + L\eta^2 E_{v_k}\big[\|e_k\|^2\big] \notag \\
        &\overset{(a)}{\leq} \tilde \Phi_\delta(u_{k}) -  E_{v_k}[\tilde \Phi_\delta(u_{k+1})] + \frac{\eta}{2}\|\nabla \tilde \Phi_\delta(u_k)\|^2 \notag \\
        &+ \frac{\eta}{2}\|E_{v_k}[e_k]\|^2 + 3L^3\eta^2\delta^2(p+4)^3 + L\eta^2 E_{v_k}\big[\|e_k\|^2\big]
    \end{align}
    where $(a)$ follows from $<a,b> \leq \|a\|\|b\| \leq \frac{1}{2} (\|a\|^2 + \|b\|^2)$. Exploiting Jensen's inequality, we have
    \begin{align} \label{eq:aux1}
        &\big( \frac{\eta}{2} - 4L\eta^2(p+4) \big) \|\nabla \tilde \Phi_\delta(u_k)\|^2 \notag \\
        &\leq \tilde \Phi_\delta(u_{k}) -  E_{v_k}[\tilde \Phi_\delta(u_{k+1})]+ 3L^3\eta^2\delta^2(p+4)^3 \notag \\
        &+ (L\eta^2 + \frac{\eta}{2})E_{v_k}\big[\|e_k\|^2\big].
    \end{align}

    Note that the left-hand side of~\eqref{eq:aux1} is positive ($0<\eta<1/8L(p+4)$). Now, we define $\xi := \big( \eta/2 - 4L\eta^2(p+4) \big)$. Combining \eqref{eq:aux1} and Lemma~\ref{lem:gaussianSmooth}($ii$), we have
    \begin{align} \label{eq:aux2}
        &\|\nabla \tilde \Phi(u_k)\|^2 \notag \\
        &\leq \frac{2}{\xi} \big(\tilde \Phi_\delta(u_k) - E_{v_k}[\tilde \Phi_\delta(u_{k+1})]\big) + \frac{6}{\xi}L^3\eta^2\delta^2(p+4)^3 \notag \\
        &+\frac{1}{\xi}(2L\eta^2+\eta)E_{v_k}\big[\|e_k\|^2\big] + \frac{1}{2} \delta^2L^2(p+6)^3.
    \end{align}
    Taking the expectations of both sides of \eqref{eq:aux2} with respect to $v_0, \dots, v_{k-1}$, we have
    \begin{align}
        &E\big[\|\nabla \tilde \Phi(u_k)\|^2 \big] \notag \\
        &\leq \frac{2}{\xi} \big(E[\tilde \Phi_\delta(u_k)] - E[\tilde \Phi_\delta(u_{k+1})]\big) + \frac{6}{\xi}L^3\eta^2\delta^2(p+4)^3 \notag \\
        &+\frac{1}{\xi}(2L\eta^2+\eta)E\big[\|e_k\|^2\big] + \frac{1}{2} \delta^2L^2(p+6)^3.
    \end{align}
    Summing up from $k = 0, \dots, T-1$, we obtain
    \begin{align} \label{eq:aux3}
        &\sum_{k=0}^{T-1} E_{v}[\|\nabla \tilde \Phi(u_{k})\|^{2}] \notag \\
        &\leq \frac{2}{\xi} \big(E[\tilde \Phi_\delta(u_0)] - E[\tilde \Phi_\delta(u_{T})]\big) + \frac{6}{\xi}TL^3\eta^2\delta^2(p+4)^3 \notag \\
        &+\frac{1}{\xi}(2L\eta^2+\eta)\sum_{k=0}^{T-1}E\big[\|e_k\|^2\big] + \frac{1}{2} T\delta^2L^2(p+6)^3 \notag \\
        &\overset{(a)}{\leq} \frac{2}{\xi} \big(\tilde \Phi_\delta(u_0) - \tilde \Phi_{low}\big) + \frac{6}{\xi}TL^3\eta^2\delta^2(p+4)^3 \notag \\
        &+\frac{1}{\xi}(2L\eta^2+\eta)\sum_{k=0}^{T-1}E\big[\|e_k\|^2\big] + \frac{1}{2} T\delta^2L^2(p+6)^3, \notag \\
        &\overset{(b)}{\leq} \frac{2}{\xi} \big(\tilde \Phi_\delta(u_0) - \tilde \Phi_{low}\big) + \frac{6}{\xi}TL^3\eta^2\delta^2(p+4)^3 \notag \\
        &+\frac{4M_\Phi^2\mu pT}{\xi \delta^2}(2L\eta^2+\eta) + \frac{1}{2} T\delta^2L^2(p+6)^3,
    \end{align}
    where $(a)$ follows the fact that $\tilde \Phi$ is bounded from below, and $(b)$ obtained from Lemma~\ref{lem:boundE}. Finally, we divide the both sides of~\eqref{eq:aux3} by $T$, and substitute $\xi$, which lead to
    \begin{align} \label{eq:aux5}
        &\frac{1}{T}\sum_{k=0}^{T-1} E_{v}[\|\nabla \tilde \Phi(u_{k})\|^{2}] \notag \\
        &\leq \frac{4\big(\tilde \Phi_\delta(u_0) - \tilde \Phi_{low}\big)}{T\eta\big(1 - 8L\eta(p+4)\big)} + \frac{12L^3\eta\delta^2(p+4)^3}{1 - 8L\eta(p+4)} \notag \\
        &+ \frac{8M_\Phi^2\mu p(2L\eta +1)}{\delta^2\big(1 - 8L\eta(p+4)\big)} + \frac{\delta^2L^2(p+6)^3}{2} \notag \\
        &\leq \mathcal{O}\bigg(\frac{1}{T\eta(1-\eta p)}\bigg) + \mathcal{O}\bigg(\frac{\eta \delta^2 p^3}{(1-\eta p)}  \bigg) + \mathcal{O}\bigg(\frac{\mu p \eta}{\delta^2 (1-\eta p)}\bigg) \notag \\
        &+ \mathcal{O}\bigg(\delta^2p^3\bigg).
    \end{align}

    Completing the proof of Theorem~\ref{thm:objConvergence}, we continue to prove Theorem~\ref{thm:optimal_parameters_selection}. First, we set $\eta = 1/16L(p+4)$ and substitute it in~\eqref{eq:aux5}. We obtain
    \begin{align} \label{eq:aux66}
        \frac{1}{T}\sum_{k=0}^{T-1} E_{v}[\|\nabla \tilde \Phi(u_{k})\|^{2}] \leq \frac{c_1}{T} + \frac{c_2 L^2 \delta^2}{2} + \frac{c_3 \mu}{\delta^2},
    \end{align}
    where
    \begin{align*}
        c_2 &= (p+6)^3 + (p+4)^2, \notag \\
        c_3 &= \frac{2M_\Phi^2 p(8p+33)}{p+4}.
    \end{align*}
    Then, we minimize the right-hand side of~\eqref{eq:aux66} with respect to $\delta$. Substituting $\delta$ with its optimal value, as given in the theorem statement, we get
    \begin{align}
        \frac{1}{T}\sum_{k=0}^{T-1} E_{v}[\|\nabla \tilde \Phi(u_{k})\|^{2}] \leq \frac{c_1}{T} + L\sqrt{2c_2 c_3 \mu} \leq \epsilon.
    \end{align}
    Solving $c_1/T \leq \epsilon/2$ for $T$ and $L\sqrt{2c_2c_3\mu} \leq \epsilon/2$ for $\mu$ give the lower bound for the number of iterations required to reach $\epsilon$-accuracy convergence and $\mu_1$, respectively. Furthermore, from Lemma~\ref{lem:gaussianSmooth}$(ii)$,
    \begin{align} \label{eq:aux6}
        \vert \obj_\delta(u) - \obj(u) \vert \leq \frac{\delta^2}{2}Lp = \frac{p\sqrt{c_3 \mu}}{\sqrt{2c_2}} \leq \epsilon_\Phi.
    \end{align}
    Solving~\eqref{eq:aux6} for $\mu$ gives $\mu_2$ and completes the proof.

\section{Simulation Results}
\label{sec:simulations}

In this section, we show the performance of the proposed Algorithm~\ref{alg:2PTRGF}.
In line with~\cite{he2023model}, we consider the following 
nonlinear system:
\begin{align} \label{eq:simSystem}
    \begin{split}
        x_{t+1} &= Ax_t + Bu_t + Ed_x \\
        &\qquad\quad+ F\big(x_t-x_{\text{ss}}(u_t,d_x)\big) \otimes \big(x_t-x_{\text{ss}}(u_t,d_x)\big) \\
        y_t &= Cx_t + Dd_y,
    \end{split}
\end{align}
where $x\in \mathbb{R}^{10}$, $u\in \mathbb{R}^{5}$, $d_x,d_y\in \mathbb{R}^5$ and $y\in \mathbb{R}^{5}$ are the state, input, disturbances, and output of the system, accordingly. Moreover, $x_{\text{ss}}(u_t,d_x) \triangleq (I-A)^{-1}(Bu_t+Ed_x)$ is the steady-state map, where $I$ is the identity matrix of the corresponding order. The final term in the state equation of \eqref{eq:simSystem} can be interpreted as the residual error when linear dynamics serve as an approximation for general nonlinear dynamics at steady state. The system matrices in~\eqref{eq:simSystem} are randomly drawn from the standard uniform distribution. Also, we let $\Vert A\Vert_2 = 0.05$ and $\Vert F\Vert_1 = 0.01$ to ensure the stability of the system. The disturbances $d_x, d_y$ are produced from the standard normal distribution. We consider the following instance of~\eqref{eq:trueProblem}:
\begin{align}\label{eq:simObj}
    \min_{u,y} ~ \Phi(u,y)= u^{\top}R_1u + R_2^{\top}u + \|y\|^2.
\end{align}
In~\eqref{eq:simObj}, we define the positive semi-definite matrix $R_1 = R_3^{\top}R_3 \in \real^{5\times 5}$. The elements of $R_2 \in \real^5$ and $R_3 \in \real^{5\times 5}$ are drawn from the standard uniform distribution. Thus, the objective in~\eqref{eq:simObj} is a smooth convex function.

\begin{figure}[t]
    \centering \subfigure[Comparison of the squared norm of the cost gradient for each method.]{\includegraphics[width=1\columnwidth]{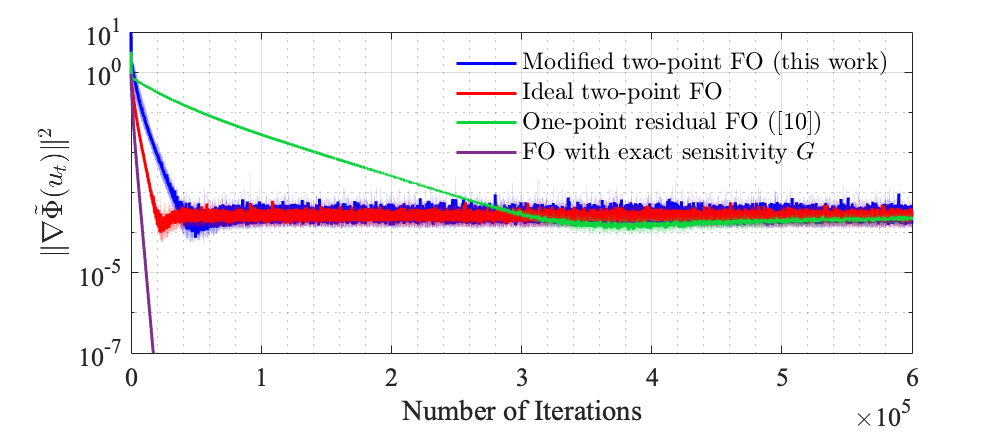}}
    \centering \subfigure[Comparison of the optimality gap for each method.]{\includegraphics[width=1\columnwidth]{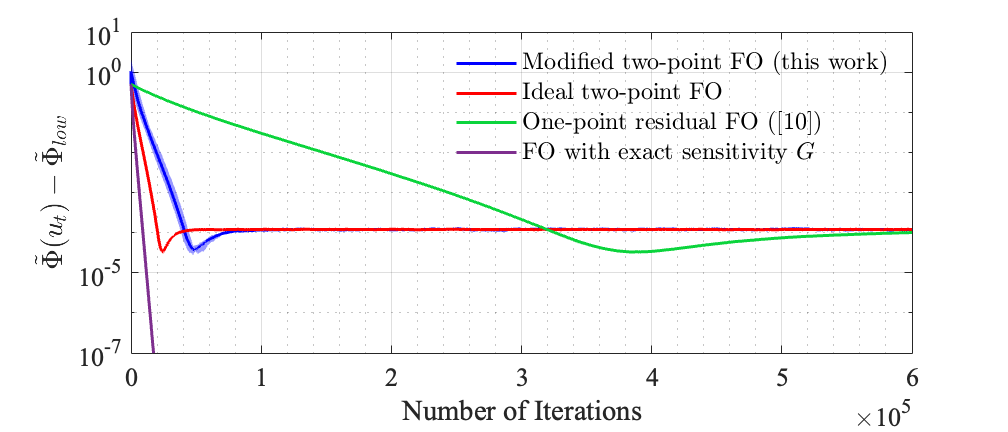}} 
    \caption{Comparison of the proposed modified two-point FO ($\eta = 40\times 10^{-5}$), idealized two-point FO ($\eta = 40\times 10^{-5}$), one-point residual proposed in~\cite{he2023model} ($\eta = 2.5\times 10^{-5}$), and classical first-order gradient descent with the exact gradient of the cost ($\eta = 100\times 10^{-5}$). $\delta$ is set to $5\times 10^{-5}$ for all methods.}
    \vspace{-.5cm}
    \label{fig:compDiffMethod}
\end{figure}

In Fig.~\ref{fig:compDiffMethod}, we propose a comparison 
between the performance of Algorithm~\ref{alg:2PTRGF}, the 
FO method with exact 
gradient~\cite{bianchin2023online}, an  idealized two-point 
method where the state of the plant is restarted at each 
iteration mimicking~\cite{YN-VP:17,AA-OD-LX:10,OS-17}, and
the one-point residual feedback 
optimization~\cite{he2023model}.
More precisely, the FO method with exact 
gradient~\cite{bianchin2023online} is given by:
\begin{align} \label{eq:FFO}
    u_{t+1} = u_t - \eta\big(\nabla_u\Phi(u_t,y_{t+1})+ G^{\top}\nabla_y\Phi(u_t,y_{t+1})\big),
\end{align}
where $G\triangleq C(I-A)^{-1}B$ is the steady-state input 
to output sensitivity of~\eqref{eq:simSystem} and $\eta > 0$ 
is the stepsize.

In Fig.~\ref{fig:compDiffMethod}, we compare the performance of the different methods using the squared norm of the gradient of the objective $\obj(u_t)$ and the optimality gap $\obj(u_t) - \obj_{\text{low}}$, where $\obj_{\text{low}}$ is the minimizer of the problem~\eqref{eq:simObj}. To obtain this figure, we used $\delta = 5 \times 10^{-5}$; the selected stepsizes are $40\times 10^{-5}$, $2.5\times 10^{-5}$, and $100\times 10^{-5}$ for the modified and idealized two-point FO controller, the one-point residual FO controller, and the first-order controller~\eqref{eq:FFO} with the exact gradient of the objective, respectively. The selected stepsizes have been optimized via trial-and-error, selecting the largest values that yield a converging algorithm. In the figures, the solid curve represents the average trajectory across $10$ experiments, whereas the shaded region illustrates the variation in these trajectories. We note that all randomized methods yield moderate oscillatory trajectories, arising from the stochastic nature of input perturbations. In contrast, the exact method~\eqref{eq:FFO}  does not exhibit this behavior because of its deterministic nature.

In Fig.~\ref{fig:compDiffMethod}(a), we observe that the feedback controller using the exact sensitivity $G$ outperforms the others; Algorithm~\ref{alg:2PTRGF} gives a solution accuracy and a convergence rate comparable to the idealized two-point but converges more slowly; the 
one-point residual feedback methods~\cite{he2023model} has a rate a convergence that is considerably worse than all other methods considered. This aligns with the existing theoretical guarantees, which shows that random two-point gradient-free methods generally exhibit faster convergence than random one-point gradient-free methods~\cite{zhang2022new}. A similar pattern is observed in Fig.~\ref{fig:compDiffMethod}(b); the first-order controller achieves the best performance, and the proposed modified controller convergence rate is close to that of the idealized two-point controller and surpasses the one-point residual controller. Moreover, the feedback mechanism of these controllers inherently reduces spikes caused by disturbances.

\begin{figure}[t]
    \centering 
    \includegraphics[width=1\columnwidth]{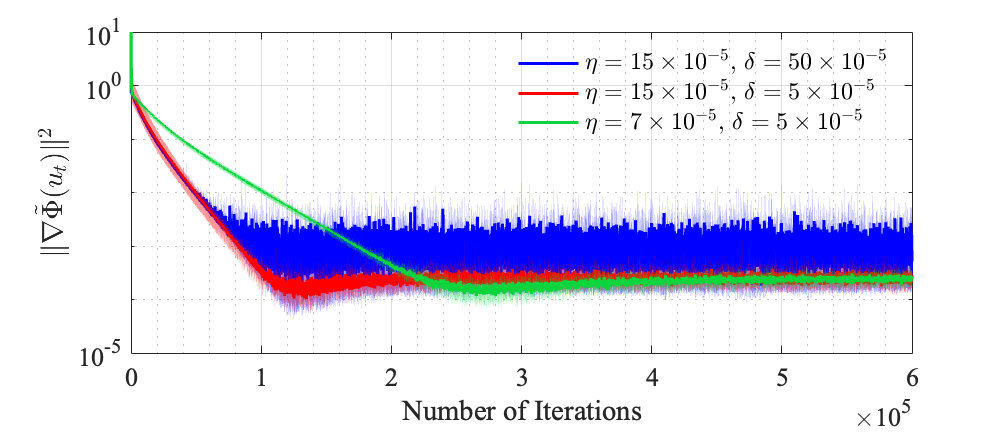}
    \caption{Performance of the proposed modified two-point FO for different stepsize ($\eta$) and smoothing parameter ($\delta$).}
    \label{fig:compDiffEtaDelta}
    \vspace{-0.5cm}
\end{figure}

Fig.~\ref{fig:compDiffEtaDelta} illustrates the impact of the 
stepsize ($\eta$) and the smoothing parameter ($\delta$). 
Decreasing the smoothing ratio $\delta$, in general, leads to 
an improvement in steady-state accuracy and reduces high-frequency variations due to large exploration steps (cf. Fig.~\ref{fig:compDiffEtaDelta} blue line and red line). However, in our simulations, we observed that exceedingly small values of $\delta$ may lead to algorithms with poor robustness when noise is involved in the sensing and actuation signals, as noise may interfere with the exploration step (step 2 of Algorithm~\ref{alg:2PTRGF}). On the other hand, increasing the stepsize ($\eta$) within the allowable range improves the rate of convergence (cf. Fig.~\ref{fig:compDiffEtaDelta} green line and red line).

\section{Conclusions} \label{sec:conclusions}

This paper introduced a two-point gradient-free feedback optimization method for controlling dynamical systems to an optimal steady-state point. Unlike traditional FO methods that rely on model-dependent gradient estimates, the proposed algorithm estimates the gradient using function evaluations, which makes it fully model-free. We modified the standard two-point zeroth-order method to make it practical for control systems. Theoretical analysis provided guarantees on the convergence of the method, and we compared the proposed algorithm with other existing zeroth-order algorithms to show its effectiveness. Further studies could extend the algorithm to a constrained optimization problem and explore adaptive stepsize to enhance the performance. In addition, relaxing the assumption on contraction properties would be an interesting research direction.

\appendix
\subsection{Zeroth-Order Optimization}

Zeroth-order or gradient-free optimization is the concept of utilizing function evaluations to estimate the gradient of the cost without the need to access the gradient directly. Zeroth-order optimization is a well-studied field; as such, the seminal work of~\cite{YN-VP:17} sets a concrete framework in this field, introducing two-point gradient estimates, which we utilize in this work. The goal is to approximate the first-order gradient of a function using only function evaluations. To achieve this, we need to perturb the function around the current point in a uniformly distributed manner across all directions that leads to considering a Gaussian-smoothed version of the function $f(u): \real^p \to \real$, as introduced by~\cite{YN-VP:17},
\begin{align} \label{eq:gaussianApprox}
    f_\delta(u) := E_{v \sim \mathcal{N}(0, I_p)}[f(u+\delta v)],
\end{align}
where the elements of the vector $v$ are i.i.d. standard Gaussian random variables. The following Lemma bounds the approximation error of the function $f_\delta(u)$, which is developed in~\cite{YN-VP:17}.
\begin{lemma}[\!\!{\cite[Theorem 1 and Lemma 4]{YN-VP:17}}] \label{lem:gaussianSmooth}
    If $f: \real^p \to \real$ is $L$-smooth, then for any $u \in \real^p$, $\delta > 0$, and $f_\delta(u)$ given in~\eqref{eq:gaussianApprox}, 
    \begin{align*}
        (i)\quad &\vert f_\delta(u) - f(u) \vert \leq \frac{\delta^2}{2}Lp, \nonumber \\
        (ii)\quad &\Vert \nabla f(u) \Vert^2 \leq 2\Vert \nabla f_\delta(u)  \Vert^2 + \frac{\delta^2}{2}L^2(p+6)^3.
    \end{align*}
    \QEDB
\end{lemma}

For an objective function $f(u): \real^p \to \real$, the gradient-free oracle proposed in~\cite{YN-VP:17} is
\begin{align} \label{eq:2ptGradEst}
    \tilde g_{\delta}(u_t) = \frac{v_t}{\delta} \big( f(u_t+\delta v_t) - f(u_t) \big)
\end{align}
where $v_t$ is a random vector of the corresponding size drawn from the standard multivariate normal distribution, and $\delta$ is a smoothing parameter, which represents the amplitude of the exploration noise. Note that the gradient estimation in \eqref{eq:2ptGradEst} requires two function evaluations at time $t$, which poses a challenge in the control setting, as only one actuation step can be applied to the system at a given time. In this work, we modify the algorithm to make it compatible with the FO setting.

\begin{lemma}[\!\!{\cite[Lemma 5]{YN-VP:17}}] \label{lem:gradEstBound}
If $f: \real^p \to \real$ is $L$-smooth, given any $u \in \real^p$, $\delta > 0$, $v\sim \mathcal{N}(0,I_p)$, $f_\delta(u)$ in~\eqref{eq:gaussianApprox}, and $\tilde g_\delta (u)$ in~\eqref{eq:2ptGradEst},
    \begin{align*}
        (i)\quad &E_v[\tilde g_\delta (u)] = \nabla f_\delta (u), \nonumber \\
        (ii)\quad &E_v[\Vert \tilde g_\delta (u) \Vert^2] \leq 4(p+4) \Vert \nabla f_\delta (u) \Vert^2 + 3\delta^2L^2(p+4)^3.
    \end{align*}
\QEDB\end{lemma}

Lemma~\ref{lem:gradEstBound} shows that the estimator~\eqref{eq:2ptGradEst} is an unbiased gradient estimate of the smoothed function $f_\delta(u)$ at $u_t$ besides a bound on the second moment of the gradient estimate $E_v[\Vert \tilde g_\delta (u) \Vert^2]$, which is used in our analysis in Section~\ref{sec:analysis}.

\bibliographystyle{IEEEtran}
\bibliography{REFs/alias,REFs/full_GB,REFs/GB,REFs/AM_bibliography}

\addtolength{\textheight}{-12cm}   

\end{document}